\documentclass[12pt]{amsart}
\usepackage{amssymb}

\newtheorem{theorem}{Theorem}[section]
\newtheorem{lemma}[theorem]{Lemma}
\newtheorem{proposition}[theorem]{Proposition}
\newtheorem{corollary}[theorem]{Corollary}

\theoremstyle{definition}

\numberwithin{equation}{section}

\DeclareMathOperator{\R}{\mathbb R}

\begin{document}

\baselineskip=15.5pt

\title[Involution of hyper-K\"ahler manifolds and branes]{Anti-holomorphic
involutive isometry of hyper-K\"ahler manifolds and branes}

\author[I. Biswas]{Indranil Biswas}

\address{School of Mathematics, Tata Institute of Fundamental
Research, Homi Bhabha Road, Bombay 400005, India}

\email{indranil@math.tifr.res.in}

\author[G. Wilkin]{Graeme Wilkin}

\address{Department of Mathematics, National University of Singapore,
Singapore 119076}

\email{graeme@nus.edu.sg}



\thanks{The first-named author is supported by the J. C. Bose Fellowship.
The second named author is supported by grant R-146-000-152-133 from the National University of Singapore}

\date{}

\begin{abstract}
We study complex Lagrangian submanifolds of a compact hyper-K\"ahler 
manifold and prove two results: (a) that an involution of a hyper-K\"ahler manifold which is
antiholomorphic with respect to one complex structure and which acts non-trivially on the corresponding
symplectic form always has a fixed point locus which is complex Lagrangian with respect to one of the other
complex structures, and (b) there exist Lagrangian submanifolds which are complex with respect to one complex
structure and are not the fixed point locus of any involution which is anti-holomorphic with respect to
one of the other complex structures.
\end{abstract}

\maketitle

\section{Introduction}

Let $(X, \omega_I, \omega_J, \omega_K, g)$ be a hyper-K\"ahler manifold. In this paper 
we study submanifolds of $X$ which are complex Lagrangian (of $B$ type) with 
respect to one of the complex structures $I, J, K$ and Lagrangian (of $A$ type) with 
respect to the K\"ahler form associated to the other two complex structures. The 
motivation for this comes from the paper of Kapustin and Witten 
\cite{KapustinWitten07}, where they study such submanifolds of the moduli space of 
Higgs bundles $\mathcal{M}_H$ over a compact Riemann surface. The most interesting 
examples are called $(B,A,A)$ branes, $(A,B,A)$ branes and $(A,A,B)$ branes (cf. 
\cite[Sec. 5.6]{KapustinWitten07}). Recently, a number of authors have constructed 
discrete families of $(A,B,A)$ branes in $\mathcal{M}_H$ via anti-holomorphic 
involutions on $\mathcal{M}_H$ (see \cite{BS}, \cite{BS2}, \cite{BG}, \cite{HWW}).

In this paper we prove two results related to complex Lagrangian submanifolds of 
hyper-K\"ahler manifolds. The first one (see Section \ref{sec:fixed-points}) 
describes the geometric structure on the fixed point locus of an anti-holomorphic 
involution.

\begin{theorem}\label{thm:fixed-point}
Let $X$ be a hyper-K\"ahler manifold of complex dimension $2d$ and let $I$ be one of
the complex structures with associated K\"ahler form $\omega_I$. Let $\sigma\,:\, X\,
\longrightarrow\, X$ be an involution such that $\sigma$ is anti-holomorphic with
respect to $I$ and that $\sigma^* \omega_I \,=\, -\omega_I$. Fix an element
$\theta\,\in\, H^0(X, \,\Omega_X^2) \setminus \{0\}$ (holomorphic
with respect to $I$) such that $\sigma^* \theta
\,=\, \bar{\theta}$ and the pointwise norm of $\theta$ is $2 \sqrt{d}$. Let
$\theta_J$ and $\theta_K$ be the holomorphic symplectic forms with respect to
$J$ and $K$ respectively. Assume that the fixed point locus
$S\,=\, X^\sigma\,\subset\, X$ is nonempty. Then
\begin{itemize}
\item the fixed point locus $S$ is a special Lagrangian submanifold with respect to
both $(\omega_I, \theta^d)$ and $(\omega_K, \theta_K^d)$, and

\item $S$ is a complex Lagrangian manifold with respect to $(J, \theta_J)$. 
\end{itemize}
\end{theorem}

The second result (see Section \ref{sec:deformations}) is that deformations of these 
examples give complex Lagrangian submanifolds that are not fixed points of any 
anti-holomorphic involution.

\begin{theorem}\label{thm:deformation}
Let $X$ be a hyper-K\"ahler manifold of complex dimension $2d$ and let
$I$ be one of the complex structures with associated K\"ahler form $\omega_I$. Let
$\sigma\,:\, X\,\longrightarrow\, X$ be an involution such that $\sigma$ is
anti-holomorphic with respect to $I$ and that $\sigma^* \omega_I\,=\, -\omega_I$. Suppose
that the fixed point set $S$ is compact and has positive first Betti number. Then there exist continuous families of complex Lagrangian submanifolds of $X$ that are not fixed points of an anti-holomorphic
involution $\sigma \,:\, X\,\longrightarrow \,X$.
\end{theorem}

It is natural to ask whether one can construct new examples of 
$(A,B,A)$ branes in the moduli space of Higgs bundles via deformations of the 
discrete families in \cite{BS}, \cite{BS2}, \cite{BG} and \cite{HWW}.

\section{The fixed point set of an anti-holomorphic involution on a hyper-K\"ahler manifold}\label{sec:fixed-points}

In this section we prove Theorem \ref{thm:fixed-point}, which describe the geometric 
structures on the fixed point locus of the involution $\sigma$.

A hyper-K\"ahler manifold is a quintuple $(X\, , g\, ,I\, ,J\, ,K)$, where
$X$ is a connected smooth manifold, $g$ is a Riemannian metric on $X$ and $I$, $J$,
$K$ are integrable almost complex structures on $X$ such that
\begin{enumerate}
\item $I$, $J$ and $K$ are orthogonal with respect to $g$,

\item the Hermitian forms for $(I\, ,g)$, $(J\, ,g)$ and
$(K\, ,g)$ are closed, and

\item $K\,= \,IJ\,=\, -JI$, $J\,=\, -IK\,=\, KI$ and $I\,=\, JK
\,=\, -KJ$.
\end{enumerate} 
(See \cite{Jo}.)
The hyper-K\"ahler manifold is called \textit{irreducible} if the holonomy
of the Levi--Civita connection corresponding to $g$ has holonomy $\text{Sp}(m/4)$,
where $m\,=\,\dim_{\mathbb R} X$.

Let $(X\, , g\, ,I\, ,J\, ,K)$ be a compact irreducible hyper-K\"ahler manifold
of real dimension $4d$. Let $\omega$ be the K\"ahler form for $(I\, ,g)$. Let
\begin{equation}\label{e1}
\sigma\, :\, X\, \longrightarrow\, X
\end{equation}
be an anti-holomorphic involution with respect to $I$ such that $$\sigma^*\omega\,=\,-
\omega\, .$$
Since $\sigma$ is
anti-holomorphic with respect to $I$, we have $\sigma^*I \,=\, -I$.
We should emphasize that a general hyper-K\"ahler manifold does not admit such
an involution. See \cite{Kh}, \cite{DIK} for examples of $K3$ surfaces admitting
such involution (see also \cite{Ni1}, \cite{Ni2}).

Henceforth, unless specified otherwise, $I$ would be taken as the complex
structure on $X$. So all holomorphic objects on $X$ are with respect to $I$
(unless specified otherwise). For example, $\Omega_X$ denotes the holomorphic
cotangent bundle of $X$ with respect to $I$.
 
Since the hyper-K\"ahler manifold $X$ is irreducible, we have
$\dim H^0(X,\, \Omega^2_X)\,=\, 1$ \cite[p. 762, Proposition 3(ii)]{Be}.
Take any nonzero element $\theta'\,\in\, H^0(X,\, \Omega^2_X)$. We have
\begin{equation}\label{e2}
\sigma^*\theta'\, =\, c\cdot \overline{\theta'}\, ,
\end{equation}
where $c\, \in\, {\mathbb C}\setminus \{0\}$. From \eqref{e2} we have
$$
\sigma^*\sigma^*\theta'\,=\, c\overline{c}\theta'\, .
$$
Since $\sigma\circ\sigma\,=\, \text{Id}_X$, this implies that
$|c|\,=\, 1$. Setting $\theta\, := \sqrt{\bar{c}} \,  \theta'$, we obtain $\sigma^*\theta\, = \overline{\theta}$.

Fix an element $\theta\, \in\, H^0(X,\, \Omega^2_X)\setminus\{0\}$ such that
\begin{enumerate}
\item $\sigma^*\theta\, =\,\overline{\theta}$, and

\item the pointwise norm of $\theta$ with respect to $g$ is $2\sqrt{d}$.
\end{enumerate}
Such a section $\theta$ exists because the holomorphic $2$-forms on $X$ are
covariant constant with respect to the Levi-Civita connection on $(X\, ,g)$.
Note that any two such choices differ by multiplication with $-1$.

\begin{lemma}\label{lem1}
The involution $\sigma$ has the following properties:
\begin{enumerate}
\item It is an isometry for the Riemannian structure $g$ on $M$.

\item $\sigma^*{\rm Re}(\theta)\,=\, {\rm Re}(\theta)$, where ${\rm Re}(\theta)$
is the real part of $\theta$.

\item $\sigma^*{\rm Im}(\theta)\,=\, - {\rm Im}(\theta)$, where ${\rm Im}(\theta)$ 
is the imaginary part of $\theta$.
\end{enumerate}
\end{lemma}

\begin{proof}
We recall that $\omega(\alpha\, ,\beta)\,=\, g(I(\alpha)\, ,\beta)$, where $\alpha$
and $\beta$ are real tangent vectors at a point of $X$.
Since $\sigma^*I\,=\, -I$ and $\sigma^*\omega\,=\,-\omega$, the first
statement follows. The remaining two statements follow from the fact that
$\sigma^*\theta\, =\,\overline{\theta}$.
\end{proof}

Let
\begin{equation}\label{e4}
S\, \,=\, X^\sigma\, \subset\, X
\end{equation}
be the subset fixed pointwise by $\sigma$. Assume that $S$ is nonempty.
This $S$ is a real manifold of dimension $2d$, but
it need not be connected. The complex dimension of $X$ being even, the involution
$\sigma$ is orientation preserving. Let $N$ be the normal bundle of $S$. As
the differential of $\sigma$ acts on $N$ as multiplication by $-1$, and the rank of
$N$ is even, we conclude that $d\sigma$ preserves the orientation of $N$. Since
$\sigma$ is orientation preserving and $d\sigma$ preserves the orientation of $N$,
it follows that $S$ is oriented.

\begin{lemma}\label{lem2}
The fixed point locus $S$ is a special Lagrangian submanifold with respect to
$(\omega\, ,\theta^d)$.
\end{lemma}

\begin{proof}
Let
$$
\iota\, :\, S\, \hookrightarrow\, X
$$
be the inclusion map. We have $\iota^*\sigma^*\omega\,=\, \iota^*\omega$
because $S$ is fixed pointwise by $\sigma$. On the other
hand, $\sigma^*\omega\,=\,-\omega$. Combining these we get that
$$
\iota^*\omega\,=\, 0\, .
$$
Therefore, $S$ is Lagrangian with respect to the K\"ahler form $\omega$.

To prove that $S$ is special Lagrangian we need to show that
\begin{equation}\label{f1}
\iota^* {\rm Im}(\theta^d)\,=\, 0\, ,
\end{equation}
where ${\rm Im}(\theta^d)$ is the imaginary part of the $(2d\, ,0)$-form
$\theta^d$.

We have $\iota^*\sigma^*\theta\,=\, \iota^*\theta$ because $S$ is fixed
pointwise by $\sigma$. On the other hand, from Lemma \ref{lem1}(3), 
$$
\iota^*{\rm Im}(\theta)\,=\, - \iota^*{\rm Im}(\theta)\, .
$$
Combining these it follows that $\iota^*{\rm Im}(\theta)\,=\, 0$. This immediately
implies that \eqref{f1} holds. So $S$ is a special Lagrangian manifold.
\end{proof}

Let $J$ be the almost complex structure on $X$ uniquely given by the equation
\begin{equation}\label{e6}
{\rm Re}(\theta)(\alpha\, ,\beta)\,=\, g(J(\alpha)\, ,\beta)\, ,
\end{equation}
where $\alpha$ and $\beta$ are real tangent vectors at a point of $X$. Similarly,
$K$ is the almost complex structure on $X$ that satisfies the equation
\begin{equation}\label{e7}
{\rm Im}(\theta)(\alpha\, ,\beta)\,=\, g(K(\alpha)\, ,\beta)\, .
\end{equation}
Both $(X\, ,J\, , g)$ and $(X\, , K\, ,g)$ are K\"ahler manifolds, in particular,
both $J$ and $K$ are integrable. Let $\omega_J$ and $\omega_K$ be the K\"ahler forms
for $(J\, , g)$ and $(K\, ,g)$ respectively. So, from \eqref{e6} and \eqref{e7},
\begin{equation}\label{e8}
{\rm Re}(\theta)(\alpha\, ,\beta)\,=\, \omega_J(\alpha\, ,\beta) \, ~ \ \text{ and }
\ ~ \, {\rm Im}(\theta)(\alpha\, ,\beta)\,=\, \omega_K(\alpha\, ,\beta)
\end{equation}
We recall that
\begin{equation}\label{e5}
\theta_J\, :=\, \omega_K+\sqrt{-1}\omega \, ~ \ \text{ and }\ ~ \,
\theta_K\, :=\, \omega_J+\sqrt{-1}\omega
\end{equation}
are holomorphic symplectic forms on the K\"ahler manifolds $(X\, , J\, , g)$
and $(X\, , K\, , g)$ respectively.

\begin{proposition}\label{prop1}
The involution $\sigma$ is holomorphic with respect to $J$, and it is
anti-holomorphic with respect to $K$. Also,
$$
\sigma^*\omega_J \,=\, \omega_J \, ~ \ \text{ and }\ ~ \,
\sigma^*\omega_K \,=\, -\omega_K\, .
$$
Furthermore,
$$
\sigma^*\theta_J \,=\, -\theta_J \, ~ \ \text{ and }\ ~ \,
\sigma^*\theta_K \,=\, \overline{\theta_K}\, ,
$$
where $\theta_J$ and $\theta_K$ are defined in \eqref{e5}.
\end{proposition}

\begin{proof}
In view of Lemma \ref{lem1}(1) and Lemma \ref{lem1}(2), from \eqref{e6}
we conclude that $\sigma$ preserves $J$. Similarly, from
Lemma \ref{lem1}(1), Lemma \ref{lem1}(2) and \eqref{e7} it follows that
$\sigma$ takes $K$ to $-K$.

{}From Lemma \ref{lem1}(2) and \eqref{e8} it follows that
$\sigma^*\omega_J \,=\, \omega_J$. From \ref{lem1}(3) and \eqref{e8} it
follows that $\sigma^*\omega_K \,=\, -\omega_K$.

Since $\sigma^*\omega\,=\, -\omega$ and $\sigma^*\omega_K\,=\, -\omega_K$, from
\eqref{e5} it follows that $\sigma^*\theta_J \,=\, -\theta_J$. Also, as
$\sigma^*\omega_J\,=\, \omega_J$, we have $\sigma^*\theta_K \,=\, \overline{\theta_K}$.
\end{proof}

\begin{corollary}\label{cor1}
The fixed point locus $S$ is a special Lagrangian submanifold with respect to
$(\omega_K\, ,(\theta_K)^d)$.
\end{corollary}

\begin{proof}
Since $\sigma$ is anti-holomorphic with respect to $K$,
$\sigma^*\omega_K \,=\, -\omega_K$ and $\sigma^*\theta_K \,=\,
\overline{\theta_K}$, the proof of Lemma \ref{lem2} goes through.
\end{proof}

\begin{corollary}\label{cor2}
The fixed point locus $S$ is a complex Lagrangian submanifold with respect to
$(J\, , \theta_J)$.
\end{corollary}

\begin{proof}
Since $\sigma$ is holomorphic with respect to $J$, it follows that $S$ is a
complex submanifold with respect to $J$. As $\sigma^*\theta_J \,=\,
-\theta_J$, it is straightforward to deduce that $S$ is a Lagrangian
submanifold with respect to the holomorphic symplectic structure $\theta_J$.
\end{proof}

\section{Deformations of complex Lagrangian submanifolds}\label{sec:deformations}

Let $(X\, , \omega^c)$ be a complex manifold of complex dimension $2d$
equipped with a holomorphic symplectic form $\omega^c$.
Let $Y \subset X$ be a complex Lagrangian submanifold (i.e., $\omega^c$ vanishes on 
$Y$). Suppose also that $X$ has a K\"ahler structure and that $Y$ is compact. In 
\cite{Hitchin99}, Hitchin studies the deformation space of compact complex Lagrangian 
submanifolds near $Y$ and shows that
\begin{enumerate}
\item the deformations are unobstructed (see also \cite{McLean98}),

\item there exists a local moduli space $M$ with real dimension equal to the first 
Betti number of $Y$, $\dim_{\R} T_{[Y]} M = b_1(Y)$, and

\item $M$ has a naturally induced special K\"ahler structure.
\end{enumerate}

For an irreducible compact hyper-K\"ahler manifold of complex dimension $2d$, we have
$TX\,=\, \Omega^{2d-1}_X$, where $TX$ is the holomorphic tangent bundle, because
$\Omega^{2d}_X$ is holomorphically trivial. Hence
$$
H^0(X,\, TX)\,=\, H^0(X,\, \Omega^{2d-1}_X)\,=\, 0
$$
(see \cite[p. 762, Proposition 3(ii)]{Be}). This implies that the group of
holomorphic automorphisms of $X$ is discrete. Any two anti-holomorphic involutions
on $X$ differ by a holomorphic automorphism, so there are at most countably many
anti-holomorphic involutions $X$.

Therefore we have the following:

\begin{theorem}
Let $(X,\omega)$ be a compact hyper-K\"ahler manifold equipped with an anti-holomorphic 
involution $\sigma$ such that $\sigma^* \omega = - \omega$. Suppose that the fixed 
point set $S$ satisfies $b_1(S) \,>\, 0$. Then there exist complex Lagrangian 
submanifolds of $X$ that are not fixed points of an anti-holomorphic involution 
$X \,\longrightarrow\, X$.
\end{theorem}

Let $X$ be a $K3$ surface with an involution $\sigma$ which is anti-holomorphic with 
respect to one of the complex structures $I$. Then there is a $\sigma$-invariant 
hyper-K\"ahler metric on $X$ and $\sigma$ is anti-symplectic with respect to 
$\omega_I$ and holomorphic with respect to one of the complex structures orthogonal 
to the original one (cf. \cite[pp. 21-22]{Donaldson90}).

Explicit examples of such anti-holomorphic involutions of $K3$ surfaces with
$b_1(S) \,>\, 0$ can be found in \cite{Kh} (see \cite[Section 3.4]{Kh} and
\cite[Section 3.8]{Kh}). Gross and Wilson have also studied such involutions on $K3$ surfaces in the context of mirror symmetry in \cite{GW}.

\section*{Acknowledgements}

We thank the referee for helpful comments. The second named author would like to thank the Tata
Institute for Fundamental Research for their hospitality while this paper was being written.

%%%%%%%%%%%%%%%%%%%%%%%%%%%%%%%%%%%%%%%%%%%%%%%%%%%%%%%%%%%%%%%%

\end{document}